\documentclass[11pt,reqno]{amsart}
\usepackage[dvipsnames]{xcolor}
\usepackage{amscd,amssymb}
\usepackage{amsmath}
\usepackage{tikz}
\usepackage{tikz-3dplot}
\usepackage{tikz-cd}

\numberwithin{equation}{section}

\newtheorem {theorem}[equation]           {Theorem}
\newtheorem* {theorem*}                   {Theorem}

\newtheorem {lemma}[equation]{Lemma}
\newtheorem {prop}[equation]     {Proposition}
\newtheorem* {prop*}     {Proposition}

\newtheorem {Claim}[equation]     {Claim}

\newtheorem {corollary}[equation]       {Corollary}

\theoremstyle{definition}
\newtheorem {defi}[equation] {Definition}
\newtheorem {Remark} [equation]         {Remark}

\newtheorem* {Example*}    {Example}

\def\R{\mathbb{R}}

\newcommand{\pp}[2]{\frac{\partial#1}{\partial#2}}

\author{Robert Cardona}\address{ Robert Cardona,
Laboratory of Geometry and Dynamical Systems, Department of Mathematics, Universitat Polit\`{e}cnica de Catalunya and BGSMath Barcelona Graduate School of
Mathematics,  Avinguda del Doctor Mara\~{n}on 44-50, 08028 , Barcelona  \it{e-mail: robert.cardona@upc.edu }
 }
 \thanks{The author acknowledges financial support from the Spanish Ministry of Economy and Competitiveness, through the Mar\'ia de Maeztu Programme for Units of Excellence in R\&D (MDM-2014-0445) via an FPI grant. The author is also supported by the grant MTM2015-69135-P/FEDER  and AGAUR grant 2017SGR932.}

\title{Steady Euler flows and Beltrami fields in high dimensions}

\begin{document}

\begin{abstract}
Using open books, we prove the existence of a non-vanishing steady solution to the Euler equations for some metric in every homotopy class of non-vanishing vector fields of any odd dimensional manifold. As a corollary, any such field can be realized in an invariant submanifold of a contact Reeb field on a sphere of high dimension. The constructed solutions are geodesible and hence of Beltrami type, and can be modified to obtain chaotic fluids. We characterize Beltrami fields in odd dimensions and show that there always exist volume-preserving Beltrami fields which are neither geodesible nor Euler flows for any metric. This contrasts with the three dimensional case, where every volume-preserving Beltrami field is a steady Euler flow for some metric. Finally, we construct a non-vanishing Beltrami field (which is not necessarily volume-preserving) without periodic orbits in every manifold of odd dimension greater than three.
\end{abstract}

\maketitle

\section{Introduction}
The Euler equations, which model the dynamics of an inviscid and incompressible fluid, are written for a general Riemannian manifold $(M,g)$ as
\begin{equation*}
\begin{cases}
\partial_t X + \nabla_X X &= -\nabla p \\
\operatorname{div}X=0
\end{cases}
\end{equation*}
where $X$ is the velocity of the fluid, the scalar function $p$ is the pressure, and the differential operators are taken with respect to the metric $g$. A topological and geometric approach happens to be very enriching for the study of steady Euler solutions (cf. \cite{AK} for a monograph on Topological Hydrodynamics). A lot of work has been done to understand such steady flows in dimension $3$ and its geometric properties, but higher dimensions are much less studied. Recently, the interest on high dimensional Euler flows has flourished both in the context of the analysis of partial differential equations \cite{T1} and in the geometrical context of steady solutions \cite{CMPP,PRT}. The aim of this paper is to study three kinds of non-vanishing vector fields and its interactions in high dimensions: steady solutions to the Euler equations, geodesible fields and Beltrami fields. Recall that a geodesible vector field is a vector field such that its orbits are geodesics for some metric. In the other hand, we define Beltrami fields in odd dimensions as vector fields which are parallel to their curl for some metric. For the study of steady solutions, the Bernoulli function plays an important role: this function is defined as $B=p+\frac{1}{2}g(X,X)$, where $p$ is the pressure, the vector field $X$ is the velocity of the fluid and $g$ is the Riemannian metric.

We start introducing a geometric structure, that we call stable Eulerisable structure, which uniquely determines a vector field, and coincides with stable Hamiltonian structures in dimension three. This vector field is a unit length geodesible volume-preserving field and we can check that such structures provide an equivalent geometric formulation (in terms of differential forms) for the study of such fields. Volume-preserving geodesible fields are in correspondence with vector fields that are solutions to the steady Euler equations for some metric and constant Bernoulli function. It follows that vector fields defined by stable Eulerisable structures are steady Euler flows of this type. This viewpoint unveils their geometric wealth and allows us to naturally import topological techniques from the contact and stable Hamiltonian world. We show that one can construct stable Eulerisable structures supported by open books, and use it to prove the existence of such structures in every homotopy class of non-vanishing vector fields of any odd dimensional manifold. In particular, by the mentioned relation with steady Euler flows, we deduce the following result.

\begin{theorem}\label{Eulerex}
Given an odd dimensional manifold and a homotopy class of non-vanishing vector fields, there exist a metric and a vector field in the given class that is a steady solution to the Euler equations with constant Bernoulli function.
\end{theorem}

The constructed solutions are geodesible and hence of Beltrami type, but are not a reparametrized Reeb field of a contact form. A more geometric interpretation of this existence theorem by saying that an odd dimensional manifold can be foliated by geodesics of some metric in any homotopy class, and furthermore the vector field whose orbits are the geodesics preserves the Riemannian volume. By means of a local modification of the constructed solutions we exhibit Euler flows that are chaotic, in the sense that there is a compact invariant set with positive topological entropy. Another consequence, which follows from results in \cite{CMPP}, is the fact  that any homotopy class of vector fields of every odd dimensional manifold can be realized in the invariant submanifold of a Reeb field in a standard contact sphere of higher dimension (see the last paragraph of Section \ref{sec:euler} for a precise statement).

In three dimensions, a very fruitful source of examples of steady Euler flows are volume-preserving Beltrami fields: volume-preserving vector fields which are parallel to their curl (for some metric). The correspondence between geodesible and Beltrami fields in three manifolds shows that steady Euler flows with constant Bernoulli function are equivalent to volume-preserving Beltrami fields. The study of Beltrami flows and its properties in high odd dimensions was already proposed in \cite{GK}, where it is proved that non integrable analytic examples of steady flows are always Beltrami fields. It is mentioned that it would be interesting to construct examples and compare its properties to the three dimensional case, we do so in this work. In fact, the constructions that lead to Theorem \ref{Eulerex}, provide a lot of examples of such Beltrami type steady Euler flows (cf. Section \ref{sec:euler}).

In the high dimensional setting the correspondence between geodesible and Beltrami fields is broken: any geodesible field is Beltrami but the converse is not true. We give a characterization of Beltrami fields and provide a construction, which uses plugs and can be done volume-preserving, of vector fields that are parallel to their curl for some metric but are not geodesible. This yields also examples of volume-preserving Beltrami fields which are not solutions to the Euler equations for any metric (i.e. it is not Eulerisable): this is highly in contrast with the situation for $3$-manifolds.

\begin{theorem}\label{nongeod}
In every manifold of dimension $2n+1>3$ and every homotopy class of non-vanishing vector fields, there is a volume-preserving Beltrami field which is not geodesible nor a solution to the Euler equations for any metric.
\end{theorem}

It was proved \cite{HT,R2} that, except in a torus bundle over $S^1$, every Beltrami field in a three manifold which is either analytic or volume-preserving has a periodic orbit. This shows that Reeb fields of stable Eulerisable structures satisfy the Weinstein conjecture in dimension three except in torus bundles over the circle. We give examples of manifolds in every dimension that admit aperiodic Reeb fields of stable Eulerisable structures, which generalize the torus bundle over $S^1$ counterexample. Even if the construction in Theorem \ref{nongeod} is done using plugs, it does not directly imply that volume-preserving Beltrami fields can be aperiodic. This is because the plug cannot be used arbitrarily: one needs to find points where the geometric information of the Beltrami field has a specific normal form.

It is natural to ask if for the class of Beltrami fields one can find examples without periodic orbits. Taking into account the mentioned constrains to insert plugs, we construct aperiodic Beltrami fields (not necessarily volume-preserving) in high dimensions. Hence the existence of periodic orbits remains open only for three dimensional smooth Beltrami fields.

\begin{theorem}\label{aperio}
Let $M$ be a closed manifold of dimensions $2n+1>3$. Then $M$ admits a (not necessarily volume-preserving) Beltrami vector field without periodic orbits.
\end{theorem}

\paragraph{\textbf{Plan of the paper.}}

In Section \ref{sec:pre} we review preliminary definitions and the state of the art of the three dimensional case. In Section \ref{sec:euler}, we define stable Eulerisable structures in order to study steady Euler flows of Beltrami type. We prove an existence theorem, deducing Theorem \ref{Eulerex} and discuss its implications. In Section \ref{sec:bel} we proceed to the study of Beltrami fields, characterizing them and proving that there are Beltrami volume-preserving vector fields which are not geodesible nor Euler flows for any metric. Finally, in Section \ref{sec:Wei}, we discuss existence of periodic orbits and construct aperiodic Beltrami fields in every manifold of dimension $2n+1>3$.

\textbf{Acknowledgements:} The author is very grateful to his advisor Eva Miranda and to Fran Presas for useful discussions and encouraging this work. The author is indebted to Daniel Peralta-Salas, who read a first draft of this work and provided insightful suggestions. Finally, the author thanks Francisco Torres de Lizaur for useful comments on a final version of the article.
\section{Eulerisable, Beltrami and geodesible fields} \label{sec:pre}
In this section we provide some background on Euler flows and other related classes of vector fields. We will first define the different kinds of vector fields and then explain their interactions. Through the rest of sections, other specific background will be introduced. If not specified, we will refer through all the paper to vector fields, flows or solutions to the Euler equations which are non-vanishing.

When we consider steady solutions to the Euler equations, one can give a reformulation of the equations using differential forms. If we denote $\alpha$ the form dual to $X$ by the metric, i.e. $\alpha= g(X,\cdot)$, and $\mu$ the Riemannian volume, the equations can be written 
\begin{equation*}
\begin{cases}
\iota_X d\alpha = -dB \\
d\iota_X\mu=0
\end{cases},
\end{equation*}
where $B:= p + \frac{1}{2}g(X,X)$ is the Bernoulli function. The second equation imposes that the vector field is volume-preserving. The equations make sense in any manifold, but we will only work on odd dimensional manifolds. Note that one needs to fix the metric to look for solutions, however one can study vector fields which are solutions to the Euler equations for some metric. Following \cite{PRT}, we introduce the notion of Eulerisable field.

\begin{defi}
Let $M$ be manifold with a volume form $\mu$. A volume-preserving vector field $X$ is \textbf{Eulerisable} if there is a metric $g$ on $M$ for which $X$ satisfies the Euler equations for some Bernoulli function $B:M \rightarrow \R$.
\end{defi}

When the preserved volume is not the Riemannian one, the equations describe the behavior of an ideal barotropic fluid. However given an Eulerisable field one can always construct a metric such that the Riemannian volume is $\mu$. In \cite{PRT} a characterization \textit{\`a la Sullivan} is given for these vector fields. However, we will only use the following simpler characterization.

\begin{lemma}[\cite{PRT}] \label{Eulerchar}
A non-vanishing volume-preserving vector field $X$ in $M$ is Eulerisable if and only if there exists a one form $\alpha$ such that $\alpha(X)>0$ and $\iota_Xd\alpha$ is exact.
\end{lemma}

Observe that if the Bernoulli function is constant, then the first Euler equation reads $\iota_X d\alpha=0$.

\begin{defi}
A vector field $X$ is \textbf{geodesible} if there exists a metric $g$ such that its orbits are geodesics.
\end{defi}

By a characterization of Gluck \cite{Gl}, a vector field is geodesible if and only if there is a one form $\alpha$ such that $\alpha(X)>0$ and $\iota_Xd\alpha=0$. When we further have $\alpha(X)=1$, then the vector field is geodesible of unit length. We might refer to $\alpha$ as the connection one form. 

The last kind of vector fields that we are interested in are Beltrami fields. Recall that for a given vector field $X$ in a Riemannian manifold $(M,g)$ of odd dimension $2n+1$, we define its curl as the only vector field $Y$ satisfying the equation
$$ \iota_Y \mu = (d\alpha)^n, $$
where $\mu$ is the Riemannian volume and $\alpha:= \iota_X g$. 
\begin{defi}
A vector field $X$ in a Riemannian  manifold of odd dimension $(M,g)$ is \textbf{Beltrami} for that metric if it is everywhere parallel to its curl, i.e.  we have $Y=fX$ for some function $f\in C^\infty(M)$.
\end{defi}

Abusing notation, we will say that a vector field is Beltrami if it is Beltrami for some metric $g$ (one can also speak of a vector field being Beltramisable). The interactions between geodesible, Beltrami and Eulerisable vector fields in three dimensions are very well understood. It was already observed in \cite{R1} that geodesible and Beltrami fields are in correspondence.

\begin{prop*}
Being geodesible and Beltrami is equivalent in three dimensions.
\end{prop*}
This correspondence leads to another one, covered partially or totally at different points of the literature.
\begin{prop*}
The following classes of vector fields are equivalent in three dimensions:
\begin{enumerate}
\item Vector fields such that for some metric they satisfy the Euler equations with constant Bernoulli function,
\item reparametrizations of Reeb fields of stable Hamiltonian structures.
\item volume-preserving geodesible fields,
\item volume-preserving Beltrami fields.
\end{enumerate}
\end{prop*}

In the next section, we consider a generalization of stable Hamiltonian structures such that the three first items are still equivalent in higher dimensions. However, we will see that volume-preserving Beltrami fields are a wider class once the dimension is at least five. In particular being Beltrami in high dimensions does not imply being geodesible, i.e. the existence of a one form positive on the Beltrami field $X$ satisfying $\iota_X d\alpha=0$.

\section{Steady Euler flows of Beltrami type}
\label{sec:euler}
In this section we provide an alternative approach to the study of geodesible volume-preserving vector fields. This will be useful to import techniques from the contact topology and stable Hamiltonian topology world and prove an existence theorem. As a consequence we obtain the proof of Theorem \ref{Eulerex}. 
\subsection{Stable Eulerisable structures}
A formulation in terms of differential forms can be given for the study of geodesible volume-preserving vector fields. This formulation opens the possibility to study these structures from a topological perspective, as done for stable Hamiltonian structure \cite{CV}.

\begin{defi}
A stable Eulerisable structure is a pair $(\alpha,\nu)$ in a manifold $M^{m+1}$, where $\alpha$ is a one form and $\nu$ a $m$-form such that 
\begin{itemize}
\item $\alpha \wedge \nu >0$,
\item $d\nu=0$,
\item  $\ker \nu \subset \ker d\alpha$.
\end{itemize}
\end{defi}
A stable Eulerisable structure defines a unit length geodesible and volume-preserving vector field $R$, defined by the equations $\iota_R\nu=0$ and $\alpha(R)=1$. We will call this vector field the Reeb vector field of $(\alpha,\nu)$, as it is done for stable Hamiltonian structures. In fact, stable Eulerisable structures are exactly the same as geodesible volume-preserving vector fields, but with some extra information fixed: the preserved volume and the connection one form (as in Gluck's characterization). Note that in three dimensions the definition coincides with the one of stable Hamiltonian structure.

Let us state a standard auxiliary lemma that will be used through the discussion. It is used for example in \cite{Gl2} or \cite{PRT}.

\begin{lemma} \label{metric}
Let $X$ be a non-vanishing vector field on a manifold $M$ and $\alpha$ a one form such that $\alpha(X)>0$. Let $\mu$ be a volume form in $M$. Then, there exist a Riemannian metric $g$ such that $g(X,\cdot)=\alpha$ and $\mu$ is the induced Riemannian volume.

\end{lemma}

\begin{proof}
Construct a metric $g$ by requiring that 
\begin{itemize}
\item $g(X,\cdot)=\alpha$,
\item $X$ is orthogonal to $\ker \alpha$,
\item arbitrary metric on $\ker \alpha$.
\end{itemize}
For such a metric, we have $g(X,Y)=\alpha(Y)$ for any $Y$ and hence $\iota_Xg=\alpha$. By taking an appropiate conformal factor in the arbitrary metric on $\ker \alpha$, we can ensure that $\mu$ is the induced Riemannian volume.
\end{proof}

In order to justify the name of a "stable Eulerisable structure", we prove a correspondence between Reeb fields of these structures and some solutions to the Euler equations for some metric. This generalizes the correspondences proved in \cite{EG}, extended in \cite{CV0} and studied in more settings \cite{CMP}.

\begin{prop}\label{corr}
In a manifold $M$ of any dimension, there is a correspondence between reparametrizations of Reeb fields of stable Eulerisable structures and solutions to the Euler equations for some metric and constant Bernoulli function.
\end{prop}

\begin{proof}
Suppose $X=fR$ is a reparametrization of a Reeb field of a stable Eulerisable structure $(\alpha,\nu)$ for some positive function $f>0$ in a manifold $M$. Using Lemma \ref{metric} construct a metric $g$ such that $\iota_Xg=\alpha$ and such that the Riemannian volume is $\mu=\frac{1}{f}\alpha \wedge \nu$. 
 
 Using that $R\in \ker \nu$ we deduce that $X$ preserves the volume $\mu$. In particular, it is a solution to the Euler equations for the metric $g$, with constant Bernoulli function.

Conversely, suppose that $X$ is a vector field satisfying the Euler equations for some metric $g$, with constant Bernoulli function. Denoting $\alpha= g(X,\cdot)$ and $\mu$ the Riemannian volume, $X$ satisfies the equations
\begin{equation*}
\begin{cases}
\iota_Xd\alpha&=0 \\
d\iota_X\mu&=0
\end{cases}.
\end{equation*}
Take $\nu:=\iota_X\mu$ and we have that $\alpha(X)=1$, $\iota_X\nu=0$, $d\nu=0$ and $\alpha \wedge \nu>0$. Hence $X$ is the Reeb field of the stable Eulerisable structure $(\alpha,\nu)$.
\end{proof}
The terminology becomes now clear, since the Reeb field of a stable Eulerisable structure is a solution to the Euler equations for some metric. The notion of being stabilized by a one form $\alpha$ comes from the world of stable Hamiltonian structures. Concretely, following \cite{CV}, it is said that a one dimensional foliation $\mathcal{L}$ is \emph{stabilizable} if there is some vector field $X$ generating $\mathcal{L}$ and some one form $\alpha$ such that $\alpha(X)=1$ and $\iota_X d\alpha=0$. 

We will restrict ourselves to odd dimensions in the next sections, where the curl operator works similarly to the three dimensional case. Then the dimension of $M$ will be $2n+1$ and $\nu$ is a form of degree $2n$. In this case, the Reeb field of a stable Eulerisable structure is Beltrami for some metric.

\begin{prop}\label{belt}
Let $R$ be the Reeb field of a stable Eulerisable structure $(\alpha,\nu)$ in a manifold $M$ of odd dimension. Then any reparametrization of $R$ is Beltrami for some metric $g$ and preserves the induced Riemannian volume. 
\end{prop}

\begin{proof}
The same metric and volume that we constructed in the first implication of Proposition \ref{corr} work. Let $X=fR$ be a reparametrization of $R$ by a positive function $f$. The curl vector field of $X$ is defined as the only vector field $Y$ such that
$$ \iota_Y \mu = (d\alpha)^n. $$ 
Then $\iota_X \iota_Y \mu= f\iota_R (d\alpha)^n=0$ since $R\in \ker d\alpha$. Hence $X$ is parallel to its curl, and again preserves $\mu$.
\end{proof}

We will prove an existence theorem for stable Eulerisable structures in odd dimensions, Theorem \ref{thm:exi}, implying Theorem \ref{Eulerex}. The relation between geodesible vector fields and open book decompositions was already suggested by Gluck \cite{Gl2}, and used in \cite{HW} to construct a geodesible vector field in any odd dimensional manifold. By using techniques coming from contact and stable Hamiltonian structures, we improve the construction to obtain the existence theorem for stable Eulerisable structures.
\subsection{Open book decompositions}
We first discuss some results on open book decompositions proved in \cite{Et}. 
\begin{defi}
An open book decomposition for a $(2n+1)$-manifold $M$ is a pair $(B,\pi)$ such that
\begin{itemize}
\item $B$ is a codimension $2$ submanifold that admits a trivial neighborhood $U=B\times D^2$,
\item $\pi:M \backslash B \rightarrow S^1$ is a fibration which, when restricted to $U$, corresponds to the projection $(p,r,\theta) \mapsto \theta$ where $(r,\theta)$ are polar coordinates in $D^2$.
\end{itemize}
We call $B$ the binding and $P:=\pi^{-1}(\theta)$ a page of the open book. If the page satisfies that it is a handlebody with handles of maximum index $n$, we say that the page is \emph{almost canonical}.
\end{defi}

Given a hyperplane field $\xi$ on the binding $B$ we can construct an hyperplane field in the whole manifold $M$. Denote by $\alpha$ a one form defining $\xi$, i.e. $\ker \alpha=\xi$. Restricting ourselves to the neighborhood $U= B \times D^2$ with polar coordinates $(r,\theta)$ in $D^2$. We define
$$ \beta=\tilde f(r)d\theta + \tilde g(r) \alpha,  $$
where $\tilde f$ and $\tilde g$ are smooth functions satisfying the following conditions. \begin{equation}\label{fg}
\begin{cases}
\tilde f(r)=r^2 \text{ near } 0, \quad \tilde f(r)=1 \text{ near } 1 \\
\tilde g(r)=1 \text{ near } 0, \quad \tilde g(r)=0 \text{ near } 1
\end{cases}.
\end{equation}  The one form $\beta$ can be extended as $\pi^*d\theta$ to the whole manifold. The kernel of $\beta$ defines an hyperplane field in $M$, whose restriction to $B$ is $\xi$. If we denote $\mathcal{H}(M)$ the hyperplane fields of a manifold, we just constructed a map 
\begin{align*}
H : \mathcal{H}(B) &\longrightarrow \mathcal{H}(M) \\
 		\xi=\ker \alpha &\longmapsto \ker \beta.
\end{align*}

Under the assumption that the pages of the open book are almost canonical, the map has a useful homotopic property.
\begin{theorem}[\cite{Et}]\label{et}
Let $M$ be any manifold of dimension $2n+1>3$ and $(B,\pi)$ any open book decomposition of $M$. Then we have
\begin{enumerate}
\item The map $H$ is well defined up to homotopy.
\item If the pages of $(B,\pi)$ are almost canonical, then the map $H$ is surjective at the homotopic level. More specifically, for any hyperplane $\eta$ in $M$ there is some hyperplane $\xi$ in $B$ such that $H(\xi)$ is homotopic to $\eta$.
\end{enumerate}
\end{theorem}

We will modify a little bit the map $H$ while keeping its properties. This modification will be useful in the construction of the stable Eulerisable structures of Theorem \ref{thm:exi}. To do so, take a smooth function $h:[0,1]\rightarrow \mathbb{R}$ such that $h(0)=0$, $h(1)=1$ and $h'(x)\neq 0$ only in $[\frac{1}{3},\frac{2}{3}]$.

\begin{lemma}\label{lemmH}
Consider the map $\tilde H:\mathcal{H}(B)\rightarrow \mathcal{H}(M)$ that sends $\xi$ to the hyperplane $\ker ( h(r) d\theta + (1-h(r))\alpha)$ in a neighborhood of the binding and extended in its complement as $\ker d\theta$.
Then $\tilde H$ also satisfies Theorem \ref{et}.
\end{lemma}

\begin{proof}
Clearly, the function $\tilde f$ is homotopic to $h$ by an homotopy $\tilde f_t=(1-t)\tilde f+th$, which satisfies $\tilde f_t(0)=0$ and $\tilde f_t(1)=1$. The function $\tilde g$ is homotopic to $1-h$ by an homotopy $\tilde g_t=(1-t)\tilde g+t(1-h)$, which satisfies $\tilde g_t(0)=1$ and $\tilde g_t(1)=0$. In particular, in the neighborhood of the binding the kernel of the one form
$$  \beta_t= \tilde f_t d\theta + \tilde h_t \alpha  $$
can be extended to the rest of the manifold as being equal to $\pi^*d\theta$. We get a family of one-forms whose kernels define an homotopy between $H(\xi)$ and $\tilde H(\xi)$. 
\end{proof}

\subsection{An existence result}

As in the case of stable Hamiltonian structures, we can relate stable Eulerisable structures to open books.

\begin{defi}
A stable Eulerisable structure $(\alpha,\nu)$ is supported by the open book $(B,\pi)$ if $\nu$ restricts as a positive volume form to each page of the open book.  If the restriction of $(\alpha,\nu)$ to any connected component of $B$ is a positive stable Eulerisable structure, we say that it is positively supported by the open book. Analogously, if the restriction to each connected component of $B$ is a negative stable Eulerisable structure, we say that it is negatively supported by the open book.
\end{defi}

In the three dimensional case, it was proved in \cite{CV} that any open book supports a stable Hamiltonian structure, and that any two stable Hamiltonian structures in the same cohomology class and same signs (induced orientations) in the binding circle components are connected by a stable homotopy supported by the open book. In constrast with these results, we prove that in a fixed open book decomponsition with almost cannonical pages, we can have stable Eulerisable structure positively supported by the open book in every homotopy class of hyperplane fields. Maybe asking that both stable Eulerisable structures induce the same stable Eulerisable homotopy class in the binding, or at least the same hyperplane field homotopy class, is enough to prove that they are connected by a stable homotopy supported by the open book. In the three dimensional case, this would correspond to the fact of inducing the same signs in the circle binding components and to Theorem 4.2 in \cite{CV}.

\begin{theorem}\label{thm:exi}
Let $M$ be a manifold of dimension $2n+1>3$ with a fixed open book decomposition $(B,\pi)$ with almost canonical pages. Then for every homotopy class of hyperplane field $[\eta]$ and every cohomology class $\gamma \in H^{2n}(M)$ there is a stable Eulerisable structure $(\alpha,\nu)$  positively supported by $(B,\pi)$ such that $\ker \alpha$ is homotopic to $\eta$ and $[\nu]=\gamma$.
\end{theorem}

\begin{proof}
Let us first prove there is some stable Eulerisable structure supported by $(B,\pi)$ in every homotopy class of hyperplane fields, delaying the discussion about the cohomology class of $\nu$.

We will prove it by induction. As a first step, Martinet-Lutz \cite{Lut, Mar} proved that in any $3$ manifold every plane field is homotopic to a contact structure $\alpha$, which is also a stable Eulerisable structure given by $(\alpha,d\alpha)$. 

Assume now that for every manifold up to dimension $2n-1$ there is a stable Eulerisable structure in every homotopy class of non-vanishing vector fields. Let $M$ be an arbitrary compact manifold of dimension $2n+1$. In \cite{Quinn}, Quinn shows that any odd dimensional manifold $M$ of dimension at least $5$ admits an open book decomposition $(B,\pi)$ such that its pages are handlebodies with handles of index less or equal than $n$. 

We are in the hypothesis of Theorem \ref{et}, and by Lemma \ref{lemmH} any homotopy class of hyperplane fields of $M$ is in the image of $\tilde H$. Hence given a class $[\eta]$, there is a hyperplane field in $B$ such that $[\tilde H(\xi)]=[\eta]$. By hypothesis there exist a stable Eulerisable structure $(\beta,\nu')$ with connection form $\beta$ on $B$ in the homotopy class of $\xi$. Let us denote by $X$ the Reeb field of $(\beta,\nu')$. The form $\beta$ defines an hyperplane field $\xi'=\ker \beta$ homotopic to $\xi$. Hence $[\tilde H(\xi')]=[\tilde H(\xi)]$ and the form $\beta$ extends in the trivial neighborhood of the binding $U=B\times D^2$ as
$$ \alpha= h(r) d\theta + (1-h(r)) \beta,  $$
where $h(r)$ was such that $h(0)=0$, $h(1)=1$ and $h'\neq 0$ only in $[\frac{1}{3},\frac{2}{3}]$. Its derivative is $d\alpha= h'dr\wedge d\theta - h' dr\wedge \beta + (1-h)d\beta $. Let us consider a vector field  of the form
\begin{equation}\label{eq:localY}
Y= f(r)\pp{}{\theta} + (1-f(r))\pi^*X,
\end{equation}
with $f$ being a smooth function $f:[0,1]\rightarrow \mathbb{R}$ satisfying that $f(0)=0$ and $f(1)=1$. To simplify notation, we will keep denoting $X$ the vector field $\pi^*X$ defined in $U$.
\begin{center}
\begin{figure}[!ht]
	\begin{tikzpicture}[scale=1]
     \node[anchor=south west,inner sep=0] at (0,0) {\includegraphics[scale=0.15]{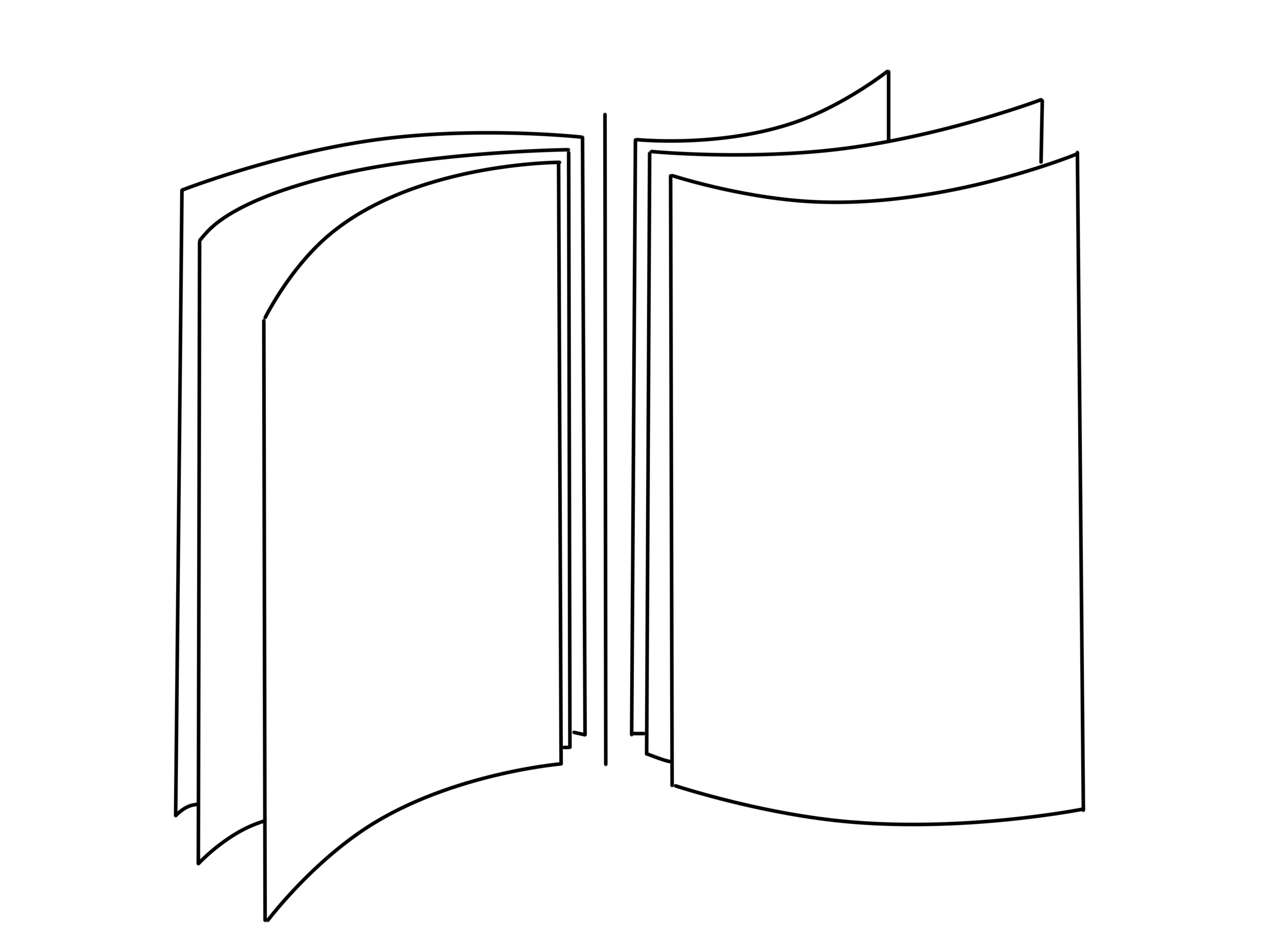}};
     
	 \node at (5,7.3) {$B$};     
     
     \draw[red,->] (4.72,3) -- (4.72,5);
     \draw[red,->] (6.1,3) arc(0:40:2);
     \draw[red,->] (7.3,3) arc(0:40:2 and 1.4);
	 \draw[red,->] (8.3,3) arc (0:40:2 and 0.8);     
     
     \draw[red,->] (3.2,3) arc(180:220:2);
     \draw[red,->] (2.4,3) arc(180:220:2 and 1.4);
     \draw[red,->] (1.8,3) arc(180:220:2 and 0.8);
     
     \draw [dashed] (5.6,3)--(5.6,4.1);
     \draw [dashed] (6.85,3)--(6.85,3.8);
     \draw [dashed] (7.85,3)--(7.85,3.4);
     
     \draw [dashed] (3.62,3)--(3.62,1.9);
     \draw [dashed] (2.9,3)--(2.9,2.2);
     \draw [dashed] (2.2,3)--(2.2,2.5);
     
     \node[red] at (9.1,3.5) {$Y$};
   \end{tikzpicture}
    
    \caption{Vector field supported by an open book}
    \label{openbook}
  \end{figure}
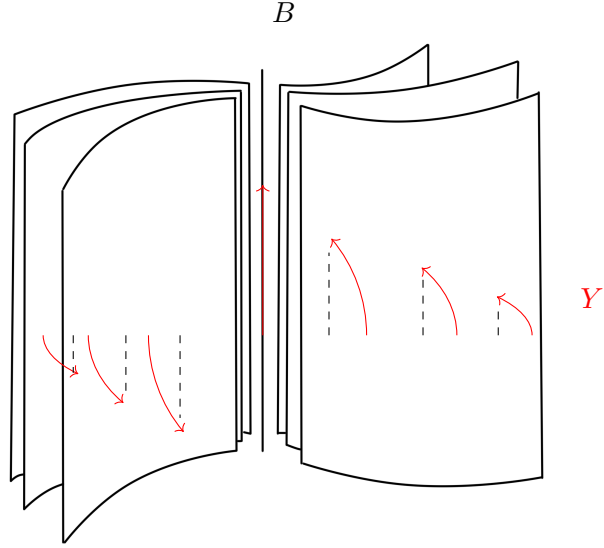
\end{center}
Imposing the geodesibility conditions $\iota_Y \alpha>0$ and $\iota_Y d\alpha=0$ we obtain the equations:
\begin{equation*}
\begin{cases}
	&hf+(1-h)(1-f)>0 \\
	&h'(1-2f)=0
\end{cases}
\end{equation*}
Take the function $f$ such that $f\neq 0$ in $(0,1]$ and $f=\frac{1}{2}$ in $[\frac{1}{3},\frac{2}{3}]$. Possible choices are depicted in Figure \ref{functs}.

\begin{center}
\begin{figure}[!ht]
\begin{tikzpicture}[scale=0.8]
      \draw[->] (-1,0) -- (14,0) node[right] {$r$};
      \draw[->] (0,-1) -- (0,3.5) node[above] {$y$};
      
      \node[black] at (0.6,1.2) {$f$};
      \node[blue] at (5.8,1.6) {$h$};
      
      \node at (12,-0.6) {$1$};    
      \node at (12,0) {$|$}; 
     
      \node at (4,-0.6) {$\frac{1}{3}$};    
      \node at (4,0) {$|$};  
      \node at (8,-0.6) {$\frac{2}{3}$};    
      \node at (8,0) {$|$}; 
           
      \draw[scale=2,domain=0:1.5,smooth,variable=\x,black] plot ({\x},{0.25*(1+tanh(5*((\x/1.5)-0.5)))});
      \draw[scale=2,domain=1.5:4.5,smooth,variable=\x,black] plot ({\x},{0.5});
	  \draw[scale=2,domain=4.5:5,smooth,variable=\x,black] plot ({\x},{0.5+0.25*(1+tanh(5*((\x-4.5)/(0.5)-0.5))});
      
      \draw[scale=2,domain=2:4,smooth,variable=\x, blue] plot ({\x},{0.5*(1+tanh(5*(((\x-2)/2-0.5)))});
      
       \draw[scale=2,domain=0:2,smooth,variable=\x, blue] plot ({\x},{0});
       \draw[scale=2,domain=4:6,smooth,variable=\x, blue] plot ({\x},{1});
    \end{tikzpicture}
    
    \caption{Possible choice of $f$ and $h$}
    \label{functs}
    \end{figure}
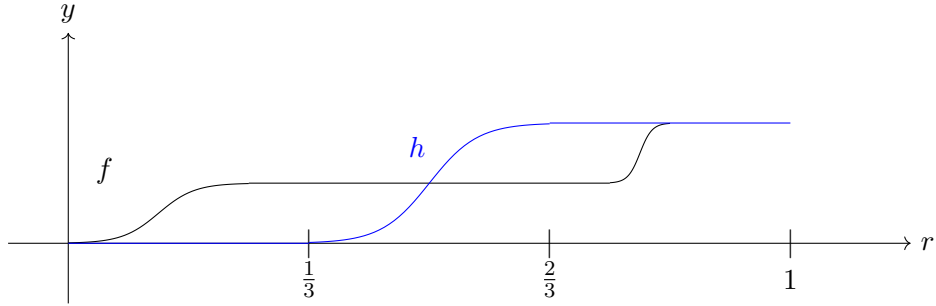
\end{center}

The vector field $Y$ can be extended as $\pi^*\pp{}{\theta}$ to the whole manifold $M$ and $\alpha$ as $\pi^*d\theta$. Hence we obtained a vector field $Y$ which is geodesible with connection form $\alpha$, and its orthogonal hyperplane field is in the class $[H(\xi')]$. 

In the trivial neighborhood $U$, there is a volume form 
$$\Theta=rdr\wedge d\theta \wedge \mu,$$
where $\mu=\beta\wedge \nu'$ is the volume form in $B$ preserved by $X$. This is a volume form in $U$ and $rdr\wedge \mu$ is a volume form in any fixed page $\{\theta=\theta_0\}$ of the trivial neighborhood $U$. Hence it can be extended to a volume form $\mu_P$ in the whole page. In particular $\Theta$ can be extended to $M$ such that away from $U$ it has the form $\pi^*d\theta \wedge \mu_P$. 

Let us check that $Y$ preserves $\Theta$. Away from $U$ this is clear, since $Y=\pi^*\pp{}{\theta}$ and hence $\mathcal{L}_Y\Theta= d(\iota_Y \Theta)=d(\mu_P)=0$. In $U$ we compute using the local expression \eqref{eq:localY} of $Y$:
\begin{align*}
\mathcal{L}_Y \Theta&= d(\iota_Y \Theta)= d(f(r)rdr\wedge \mu + (1-f(r))r\iota_X\mu \wedge dr\wedge d\theta) \\
&= 0 + (1-f(r))r(d\iota_X\mu) \wedge dr\wedge d\theta \\
&= 0,
\end{align*}
where we used that $d\iota_X \mu=d\nu'=0$, since $X$ is the Reeb field of $(\beta,\nu')$. Taking $\nu=\iota_Y \Theta$, the pair $(\alpha,\nu)$ defines a stable Eulerisable structure supported by the open book $(B,\pi)$. Since $\eta$ was arbitrary, this proves that any homotopy class of hyperplane fields contains a stable Eulerisable structure supported by the open book.
 
\hfill \newline
It remains to check that $\nu$ can be modified to obtain any homology class in $H^{2n}(M)$. This works as the three dimensional case for stable Hamiltonian structures. The open book decomposition $(B,\pi)$ can be seen as a mapping torus of a $2n$-manifold with boundary, the page $P$, with diffeomorphism given by the monodromy map of the open book. Denote $(P,\partial P)$ the page with its boundary and $(P_\varphi, \partial P_\varphi)$ the associated mapping torus. The exact sequence of the pair $(M,B)$ is
$$ H^{2n-1}(B) \xrightarrow{d*} H^{2n}(M,B) \xrightarrow{j*} H^{2n}(M)\rightarrow H^{2n}(B)=0. $$
The space $H^{2n}(M,B)$ is the same as $H^{2n}(P_\varphi,\partial P_\varphi)$. As in \cite[Lemma 4.4]{CV}, we have the following lemma in any dimension with the same proof. As before, we denote by $(r,\theta)$ coordinates in the disk component of the trivial neighborhood $U=B\times D^2$.

\begin{lemma}
Any De Rham cohomology class $\eta \in H^{2n}(M)$ has a representative of the form $\pi^*d\theta\wedge \beta$ where $\beta$ is a $(2n-1)$-form with compact support in $M\backslash B$.

\end{lemma}

The cohomology class $\eta-[\nu]$ can be represented by $\pi^*d\theta \wedge \beta$, and we can assume $\beta$ is with support in $M\backslash U$ i.e. vanishing in the trivial neighborhood $U=B\times D^2$. Defining the closed $2n$-form
$$ \tilde \nu= \nu + \pi^*d\theta \wedge \beta, $$
it satisfies that its restriction to the pages coincides with $\nu$. Hence $\tilde \nu$ is a positive volume form in the pages and represents the cohomology class $\eta$. Furthermore, in the support of $\beta$ the form $\alpha$ is given by $\pi^*d\theta$ by construction, since it is away from the neighborhood $U$. In particular $d\alpha=0$, and we have necessarily that $\ker \nu \subset \ker d\alpha$. Furthermore $\alpha \wedge \tilde \nu= \pi^*d\theta \wedge (\nu + \pi^*d\theta\wedge \beta)= \alpha \wedge \nu>0$ and so $(\alpha, \tilde \nu)$ defines a stable Eulerisable structure, positively supported by $(B,\pi)$. This concludes the proof.
\end{proof}

The homotopy classes of hyperplane fields defined by the kernel of the one form of a stable Eulerisable structure are in correspondence with the homotopy classes of non-vanishing vector fields defined by the Reeb field. From a dynamical viewpoint, we can interpret the result: every non-vanishing vector field is homotopic through non-vanishing vector fields to a geodesible and volume-preserving field.

Combining Theorem \ref{thm:exi} with Proposition \ref{corr}, we obtain the existence result for Euler steady flows Theorem \ref{Eulerex}, generalizing results in \cite{EG} to higher dimensions. Another interpretation, in terms of foliation theory is the following.

\begin{corollary}
Let $M$ be an odd dimensional manifold. In an arbitrary homotopy class of non-vanishing vector fields, there exists a metric such that $M$ can be foliated by geodesics.
\end{corollary}

It follows from the construction of the proof of Theorem \ref{thm:exi} that given a stable Eulerisable structure in the binding of any open book, we can construct one in the ambient manifold supported by the open book. Observe this holds for any open book: the hypothesis on the pages of the open book in the last Theorem is only used to show that there is a stable Eulerisable structure in every homotopy class of hyperplane fields.
\begin{corollary}\label{adOB}
Given an open book decomposition $(B,\pi)$ of $M$, a manifold of odd dimension, and a stable Eulerisable structure $(\alpha_B,\nu_B)$ on the binding, there is a stable Eulerisable structure positively supported by $(B,\pi)$ inducing $(\alpha_B,\nu_B)$ on the binding.

\end{corollary}
This provides even more examples of steady solutions to the Euler equations, and also of geodesible and Beltrami volume-preserving fields. 
\begin{Remark}
Corollary \ref{adOB} holds also for even dimensional manifolds, so any open book decomposition of an even dimensional manifold, whose binding admits a stable Eulerisable structure, also admits a stable Eulerisable structure. Since the $2$-torus admits trivially such a structure, we deduce that any four manifold admitting an open book decomposition with torus binding components admits also a stable Eulerisable structure. These type of open book decompositions were considered for example in \cite{CPV}.
\end{Remark}

In \cite{CMPP} the authors prove that given a geodesible field in a manifold $M$, it can be "Reeb embedded" in the standard sphere of dimension $3\dim M+2$. This means that for a given geodesible field $X$ in $M$, there exists an embedding $e:M\rightarrow (S^{3\dim M+2},\xi_{std})$ such that there is a contact form whose Reeb field $R$ satisfies $e_*X=R$. Since the constructed steady flows are geodesible, we deduce the following corollary.

\begin{corollary}
Any homotopy class of non-vanishing vector fields of a manifold $M$ of dimension $2n+1$ can be realized as an invariant submanifold of the Reeb field of a contact form defining the standard contact structure in the sphere $(S^{6n+5},\xi_{std})$.
\end{corollary}

\subsection{Chaotic steady Euler flows}

In \cite{GK} it is proved that in the analytic case, chaotic solutions to the Euler equations are always of Beltrami type. However, the construction of such chaotic Euler flows is left as a question.  In \cite{Gh}, Ghrist gives the first examples: if one takes a Reeb field of a contact form, which is an Eulerisable flow, one can locally modify it to obtain another chaotic Reeb field of a contact form and hence a steady Euler flow. By chaotic we mean such that there is a compact invariant set of the Reeb field possessing positive topological entropy. This proves that any contact manifold in any dimension admits a non integrable steady Euler flow. Adapting Ghrist's result to the solutions constructed in Theorem \ref{Eulerex}, we can generalize it to prove that any odd dimensional manifold admits such non integrable flows. Let us recall the contact case.

Given a contact form $\alpha$ in $M$ of dimensions $2n+1$, in the neighborhood $U=D^{2n+1}$ of any point $p$ the contact form has the expression
$$\alpha= dz + \sum_{i=1}^{n} x_i dy_i,$$
where $(z,x_1,y_1,...,x_n,y_n)$ are coordinates in the neighborhood $U$. The following theorem shows the existence of a contact chaotic Reeb field whose contact form coincides with the standard one away of a compact subset of $\R^{2n+1}$ .

\begin{theorem*}[\cite{EG2,Gh}]
There is function $F(x_i,y_i,z)$ which is equal to $1$ away of a neighborhood of $0\in D^n$ such that 
$$\lambda= F.(dz+\sum_{i=1}^n x_i dy_i),$$
is a contact form in $\R^{2n+1}$ such that there is a compact invariant set of the Reeb field possessing positive topological entropy i.e. the Reeb flow is "chaotic".
\end{theorem*}

We obtain the corollary below by combining this "inserted field" with the construction in Theorem \ref{thm:exi}.

\begin{corollary}
Every odd dimensional manifold admits a chaotic non-vanishing solution to the Euler equations for some metric.
\end{corollary}

\begin{proof}
The three dimensional case is covered, since any three manifold is contact. Let $M$ be a manifold of dimension $2n+1>3$ and an almost canonical open book decomposition $(B,\pi)$ on it. Using Theorem \ref{thm:exi}, we construct a stable Eulerisable structure $(\alpha, \nu)$ supported by $(B,\pi)$. By construction, in the trivial neighborhood of the binding $U=B\times D^2$ the Reeb vector field $X$ and the form $\alpha$ have the expressions
\begin{equation*}
\begin{cases}
 X&= f(r) \pp{}{\theta} + (1-f(r)) \pi^*X \\
 \alpha&= h(r)d\theta + (1-h(r))\beta
\end{cases}.
\end{equation*}
In particular, when $r>2/3$ we have $h=0$ and we can pick $f$ such that in a neighborhood $r\in (1-\varepsilon,1)$ we have $f=1$ (as the one depicted in Figure \ref{functs}). In particular there is a neighborhood of the form $V=S^1 \times D^{2n}$ with coordinates $(\theta,x_1,...,x_{2n})$ where $X|_V=\pp{}{\theta}$ and $\alpha|_V=d\theta$. Consider the function $r=\sqrt{\sum_{i=1}^{2n} x_i^{2}}$ and a function $\varphi(r)$ which is $r^2$ in a neighborhood of $r=0$ and vanishes in a neighborhood of $1$. Then we can change $\alpha$ by 
$$ \tilde \alpha= d\theta + \varphi(r) \alpha_{std}, $$
where $\alpha_{std}$ is the standard contact structure in the sphere $S^{2n-1}$ seen in $D^{2n}$.  In a neighborhood of $r=0$ we have that $\tilde \alpha$ is a contact form with Reeb field equal to $X$. It is easy to check that the vector field $X=\pp{}{\theta}$ still satisfies $\iota_X d\tilde \alpha=0$. Hence $X$ is also the Reeb field of the stable Eulerisable structure defined by $(\tilde \alpha, \nu)$. However, using Darboux theorem we can now find coordinates $(z,x_i,y_i)$ at a neighborhood $D$ of a point where $\tilde \alpha$ is a contact form  such that $X|_D=\partial_z$ and $\tilde \alpha|_D= dz + \sum_{i=1}^{n} x_i dy_i$. Inserting a contact form as the one in the previous Theorem yields a one form $\lambda$ which is contact, defines a Reeb field $R$ such that $\lambda$ coincides with $\tilde \alpha$ and $R$ with $X$ away of a small neighborhood of the point $p$. Extending $\lambda$ as $\tilde \alpha$ and $R$ as $X$ in the rest of the manifold we obtain a vector field $R$ which preserves a volume $\mu$  and such that $\iota_R \lambda=1$ and $\iota_R d\lambda=0$. Hence $R$ is the Reeb field of the stable Eulerisable structure $(\lambda, \iota_R \mu)$ and defines a steady solution to the Euler equations for some metric. Furthermore, the vector field is chaotic.
\end{proof}
The geometric formulation of geodesible volume-preserving vector fields allowed again to naturally import techniques coming from contact topology.

\section{High dimensional Beltrami fields}\label{sec:bel}
We study in this section the interactions between Beltrami, geodesible and Eulerisable fields. We construct vector fields which are Beltrami for some metric but that are not geodesible, in any odd dimensional manifold of dimension at least $5$. The construction, which uses plugs, can be done volume-preserving and yields examples of volume-preserving Beltrami fields which are not Eulerisable.

One can characterize vector fields which are Beltrami in a similar way as Gluck's characterization for geodesible fields.

\begin{lemma} \label{beltchar}
Let $M$ be a manifold of dimension $2n+1$. A vector field $X$ is a Beltrami field for some metric $g$ if and only if there is a one form $\alpha$ such that $\alpha(X)>0$ and $\iota_X (d\alpha)^n=0$. If furthermore $X$ preserves a volume $\mu$, one can construct a metric $g$ such that $\mu$ is the Riemannian volume.

\end{lemma}

\begin{proof}
Suppose there is such a one form. Using Lemma \ref{metric}, construct a metric $g$ such that $\iota_X g=\alpha$ and $\mu$ is the Riemannian volume. Then the vorticity field of $X$, denoted $Y$, satisfies 
$$ \iota_Y \mu = (d\alpha)^n, $$
where $\mu$ is the Riemannian volume. Contracting with $X$ and using the hypothesis we obtain that $\iota_X \iota_Y \mu=0$, which implies that $X$ is parallel to its curl and hence is Beltrami.

Conversely, if $X$ is parallel to its curl $Y$ we have that $\iota_X (\iota_Y \mu)=0$ and  $\iota_X (\iota_Y \mu)=\iota_X (d\alpha)^n$
where $\alpha=\iota_X g$.
\end{proof}

\subsection{Wilson plugs and obstructions}

Let us start by recalling Wilson's plug \cite{Wil}, used to prove the existence of non-vanishing vector fields without periodic orbits in $S^{2n+1}$ with $n>1$. We will follow the description in \cite{PPP}. \newline

\paragraph{\textbf{Standard Wilson's plug.}} We consider the manifold $W= [-2,2] \times \mathbb{T}^2 \times [-2,2] \times D^{n-4}$, and put coordinates $(z, \varphi_1, \varphi_2, r, x_1,...,x_{n-4})$. The manifold $W$ is embedded into $\R^n$ by a map $i: W \longrightarrow \R^n$ sending a point $p$ to 
$$(z, \cos{\varphi_1} (6+ (3+ r)\cos{\varphi_2}), \sin{\varphi_1} (6+ (3+r) \cos{\varphi_2}), (3+r)\sin{\varphi_2},x_1,...,x_{n-4}).$$

Let us denote $\mathbf{x}= (x_1,...,x_{n-4})$. We consider a vector field $X_W$ in $W$ with expression
$$X_W= f(z,r, \mathbf{x}) \left(\pp{}{\varphi_1} + b \pp{}{\varphi_2}\right) + g(z,r,\mathbf{x}) \pp{}{z}. $$

Choosing $b$ an irrational number and $f,g$ satisfying the following properties ensures that $W$ is a plug trapping the orbits entering through $\{z=-2, |r|\leq 1, |\mathbf{x}| \leq 1/2\}$. The properties satisfied by $f$ and $g$ are

\begin{itemize}
\item $f$ is skewsymmetric and $g$ is symmetric with respect to the $z$ coordinate,
\item $g\equiv 1$, $f \equiv 0$ close to the boundary of $W$,
\item $g \geq 0$ everywhere and vanishes only in $\{ |z|=1, |r| \leq 1, |\mathbf{x}|\leq 1/2 \}$,
\item $f\equiv 1$ in $\{ z \in [-3/2, -1/2], |r|\leq 1, |\mathbf{x}| \leq 1/2 \}$.
\end{itemize}

The same plug can be done using a manifold of the form $\tilde W= [-2,2] \times T^{n-2} \times [-2,2]$, and the trapped orbits wind around some components of the torus. The plug exist also in dimension three, however the invariant set is a circle that creates a new periodic orbit.

It is well known \cite{Gi} that Wilson's plug can be done volume-preserving, providing volume-preserving counterexamples to the generalized Seifert conjecture in $S^{2n+1}$ for $n\geq 2$. The first construction of this volume-preserving plug is due to G. Kuperberg \cite{Ku}. \newline

\paragraph{\textbf{Volume-preserving Wilson's plug.}}

In \cite[Section 2.3.1]{R1}, the explicit construction is done for three dimensions. Omitting details, let us recall the construction and give a explicit coordinate description for the case of any dimension.

Consider the manifold $P=T^{n-2}\times [1,2]\times [-1,1]$, endowed with coordinates $(\theta_1,...,\theta_{n-2},r,z)$. The first step is constructing a vector field of the form
$$X_P= H_1(r,z) + f(r,z)(\pp{}{\theta_1}+b\pp{}{\theta_2})$$
with $b$ an irrational constant number. Taking suitable functions makes $P$ a volume-preserving semi-plug (meaning that the entry and exit region do not coincide), which traps a set of zero measure isomorphic to $T^{n-2}$. This is done by taking the vector field $H_1= h_1(r,z) \pp{}{z} + h_2(r,z) \pp{}{r}$ such that $\iota_{H_1}\mu=dh$ for some volume $\mu$ of $[1,2]\times [-1,1]$ and function $h$. Taking a suitable $h$, the flow lines of $H_1$ look like in the Figure \ref{figH}, with a single singularity.
\begin{center}
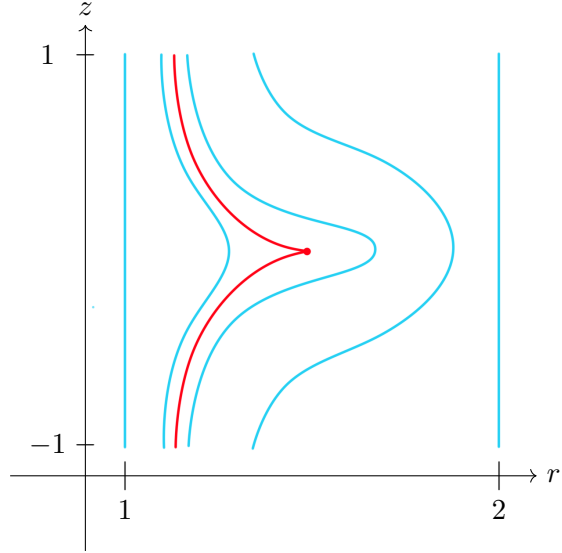
\begin{figure}[!ht]
	\begin{tikzpicture}[scale=1]
     \node[anchor=south west,inner sep=0] at (0,0) {\includegraphics[scale=0.22]{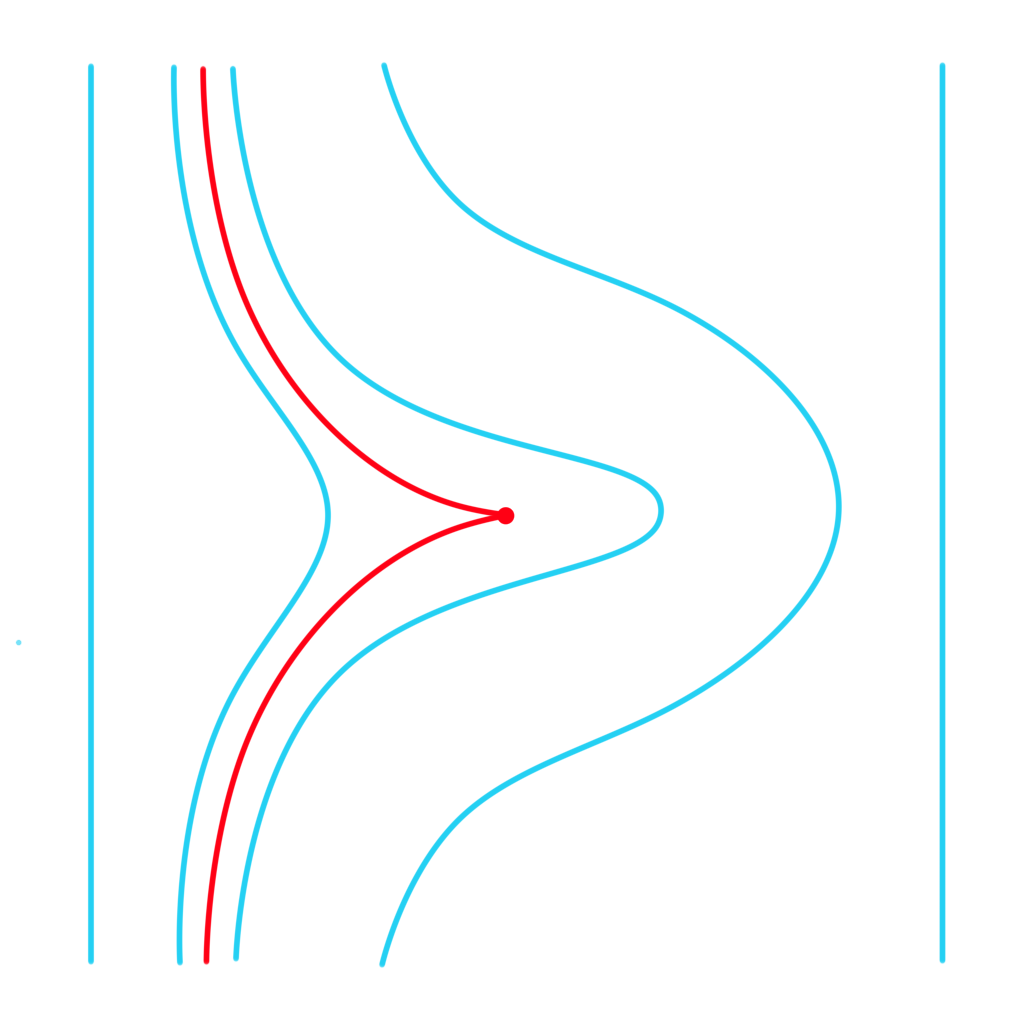}};
     
           \draw[->] (-1,0) -- (6,0) node[right] {$r$};
      \draw[->] (0,-1) -- (0,6) node[above] {$z$};
      
      \node at (0.53,-0.45) {$1$};
      \node at (0.53,0) {$|$};
      
      \node at (5.5,-0.45) {$2$};
      \node at (5.5,0) {$|$};
      
      \node at (-0.5,0.41) {$-1$};
      \node at (0,0.41) {$-$};
      
      \node at (-0.5,5.6) {$1$};
      \node at (0,5.6) {$-$};

   \end{tikzpicture}
\caption{Flow lines of $H$}
\label{figH}
\end{figure}
\end{center}
Choosing the function  $f:[1,2]\times [-1,1]\rightarrow \R^+$ such that it is zero on the boundary and positive at the singularity of $H$ is enough to make $P$ a semi-plug.  Taking the mirror-image to exchange the entry and exit regions and rescaling so that it fits $P$ yields a volume-preserving plug in $P$. Observe that again in coordinates $(r,z,\theta_1,...,\theta_{n-2})$ the constructed vector field in the plug (that we denote again $X_P$) is of the form 
\begin{equation} \label{plug}
X_P= \tilde h_1(r,z) \pp{}{z} + \tilde h(r,z) \pp{}{r} + \tilde f(r,z)(\pp{}{\theta_1}+b\pp{}{\theta_2}),
\end{equation}
and in the boundary we have $X_P|_{\partial P}= \pp{}{z}$. The function $\tilde h_1$ is positive everywhere except in the singularity, where it vanishes, but then $\tilde f$ is non-vanishing. The preserved volume is $\mu\wedge d\theta_1 \wedge ... \wedge d\theta_{n-2}$. \newline

\paragraph{\textbf{Obstructions to plugs}}
Sullivan's characterization of geodesible vector fields \cite{Sul} was used to prove that a vector field admitting a plug is not geodesible.

\begin{theorem*}[\cite{R1,PRT}]
Plugs are not geodesible in any dimension.
\end{theorem*}
In \cite{PRT} the result was obtained for the class of Eulerisable fields.

\begin{theorem*}[\cite{PRT}]
Plugs are not Eulerisable in any dimension.
\end{theorem*}

\subsection{Beltrami fields admitting plugs}
We proceed to construct a volume-preserving Beltrami field which admits the Wilson volume-preserving plug, and hence cannot be geodesible nor Eulerisable.

\begin{theorem} \label{main2}
There are volume-preserving Beltrami fields in any manifold of dimension $2n+1>3$ and any homotopy class of non-vanishing vector fields which are not geodesible nor Eulerisable.
\end{theorem}

\begin{proof}
Consider $M$ any odd dimensional manifold of dimension $2n+1\geq 5$.

 Applying Theorem \ref{thm:exi}, we know it admits a stable Eulerisable structure $(\alpha,\nu)$ with geodesible Reeb field $X$ in an arbitrary homotopy class of non-vanishing vector fields. If we denote $(B,\pi)$ the open book decomposition we used to construct the structure, by construction there are points $p$ outside the trivial neighborhood of $B$ where the vector field and its connection form are $\pp{}{\theta}$ and $d\theta$. 

In a small neighborhood $U\cong \R^n$ we can take coordinates $(z, y_1,...,y_{2n})$ such that $X|_U= \pp{}{z}$ and $\alpha|_U= dz $.

Consider the manifold $P=T^{n-2}\times [1,2]\times [-1,1]$ of the previous section with its coordinates $(z,r, \theta_1,..., \theta_{n-2})$ and vector field of equation \eqref{plug}, of the form $X_P= h_1(r,z) \pp{}{z} + h_2(r,z) \pp{}{r} +  f(r,z)(\pp{}{\theta_1}+b\pp{}{\theta_2})$, where we have omitted the tildes of the functions to simplify the notation. Take the form 
$$ \alpha_P= f(z,r) d\theta_1 +  h_1(z,r)dz. $$
It satisfies $\alpha_P (X_P)= f^2 + h_1^2 >0$ at every point since the only points where $h_1$ vanishes, $f$ does not. In a neighborhood of the boundary of $P$, the form $\alpha_P$ coincides with $dz$. This implies that once $P$ is embedded in the neighborhood $U$, both the field $X_P$ and the form $\alpha_P$ can be extended as $X$ and $\alpha$ outside the embedded copy of $P$. By standard arguments (cf. \cite[page 78]{R1}) one can make sure that the volume preserved by $X_P$ coincides in the boundary of $P$ with $\alpha\wedge \nu$. Denote $\tilde X$ and $\tilde \alpha$ the vector field and one form of the plug extended as $X$ and $\alpha$ outside of it. If we further denote $\mu$ the volume given by extending the volume in the plug as $\alpha\wedge \nu$ outside of it, it is clear that $\tilde X$ preserves $\mu$.

Let us check that $\tilde X$ is, in addition to volume-preserving, a Beltrami field. Outside the embedded copy of $P$, we have $\tilde X=X$ and $\tilde \alpha=\alpha$. Hence we have
\begin{equation}\label{geod}
\iota_{X}d\alpha=0,
\end{equation}
which trivially implies $\iota_X (d\alpha)^n=0$. Inside $P$, we have that $\tilde \alpha= \alpha_P$. By looking the coordinate description of $\alpha_P$, we have
$$ d\alpha_P= \pp{f}{z} dz \wedge d\theta_1 + \pp{f}{r} dr \wedge d\theta_1 + \pp{h_1}{r}dr \wedge dz .$$
Hence $(d\alpha_P)^2=0$. Using Lemma \ref{beltchar} we deduce that $\tilde X$ is a volume-preserving Beltrami field for some metric. Crearly in $P$ the vector field cannot be geodesible, since we know plugs are not geodesible. The fact that $(d\alpha_P)^2=0$ is not a contradiction with the fact that $\tilde X$ is not geodesible. Computing the contraction of $\tilde X$ with $d\alpha_P$ we obtain
\begin{align*}
\iota_{\tilde X}d\alpha_P= (h_1(r,z)\pp{f}{z}+h_2(r,z)\pp{f}{r})  d\theta_1 + (h_2(r,z)\pp{h_1}{r}- f\pp{f}{z}) dz + (-f\pp{f}{r}-h_1(r,z)\pp{h_1}{r})dr  ,
\end{align*}
which is clearly not constantly zero. 
\end{proof}

\begin{Remark}
One can also use the standard Wilson's plug in theorem \ref{main2}.  It is not volume-preserving and so one constructs only an example of a Beltrami field (not volume-preserving) admitting a plug. It is not geodesible, and traps a set of orbits of positive measure (in the boundary of the plug).

\end{Remark}

Combining it with the obstruction to the existence of plugs, we deduce Theorem \ref{nongeod}. Furthermore, these vector fields cannot be "Reeb-embedded" in the sense of \cite{CMPP} into any contact manifold.

\begin{Remark}
In fact one can say even more. The vector fields produced by Theorem \ref{main2} cannot be embedded in any other manifold such that $X$ extends to an Eulerisable vector field.
Let $M$ be an odd dimensional manifold and $X$ a non geodesible Beltrami volume-preserving vector field. Assume that $M$ is embedded in a manifold $N$, where there is an Eulerisable vector field $Y$ such that $Y|_M=X$. Since $Y$ is Eulerisable, there is a one form $\alpha$ such that $\iota_Y \alpha>0$ and $\iota_Y d\alpha$ is exact. If we denote $e: M \rightarrow N$ the embedding, we have that the one form $e^*\alpha \in \Omega^1(M)$ satisfies that $e^*\alpha(X)>0$ and $\iota_X d(e^*\alpha)$ is exact. But then $X$ is Eulerisable in $M$ by Lemma \ref{Eulerchar}, which is a contradiction.
\end{Remark}

\subsection{Other remarks}

As additional observations, we present another source of examples of Beltrami fields and a property concerning the relation with geodesibility.

\begin{Example*}
Let $M$ be an odd dimensional manifold. If $\mathcal{F}$ is a codimension one foliation, $M$ admits a Beltrami field transverse to it. If furthermore the foliation was minimal, the Beltrami field is volume-preserving. Let us just explain the case when $\mathcal{F}$ is minimal. Denote $\alpha$ the defining form of $\mathcal{F}$. By a result of Sullivan \cite{S2}, there is a $2n$-form $\omega$ which is closed and positive in the leaves.  The vector field defined by $\iota_X\omega=0$ and $\alpha(X)=1$ preserves the volume form $\alpha\wedge \omega$. Observe that since $\alpha$ defines a foliation, we have $\alpha \wedge d\alpha=0$ implying $(d\alpha)^2=0$. By Lemma \ref{metric}, one can construct a metric such that $X$ is parallel to its curl and the Riemannian volume is $\alpha \wedge \omega$. These examples are irrotational, since their curl is vanishing. This follows from the fact that $(d\alpha)^2=0$.
\end{Example*}

Volume-preserving examples that arise from minimal foliations exist in every odd-dimensional manifold following the results in \cite{Mei}. The following observation was suggested by Daniel Peralta-Salas.

\begin{prop}
Let $X$ be a Beltrami field in a manifold $M$ of dimension $2n+1>3$.  If $\alpha=\iota_Xg$ is generic, in the sense that $d\alpha$ is of maximal rank almost everywhere, then $X$ is geodesible.
\end{prop}

\begin{proof}
The curl of $X$ satisfies $\iota_Y \mu =(d\alpha)^n$ and since $X$ is Beltrami we know that $Y=fX$ for $f\in C^\infty(M)$. In particular, we deduce that $f\iota_X \mu=(d\alpha)^n$ and $f$ is non-vanishing almost everywhere by the genericity assumption. By contracting this equation with $X$, it follows that $X$ is in the kernel of $(d\alpha)^n$. Since $(d\alpha)^n$ is of maximal rank almost everywhere, it follows that $X$ is in the kernel of $d\alpha$ almost everywhere. Hence $\iota_Xd\alpha$ vanishes almost everywhere and by continuity $\iota_Xd\alpha\equiv 0$. By Gluck's characterization, we have that $X$ is geodesible.
\end{proof}

\section{Periodic orbits} \label{sec:Wei}

The constructed plug in Theorem \ref{main2} cannot be immediately used to prove the existence of Beltrami fields (volume-preserving or not) without periodic orbits in arbitrary manifolds. This is because the plug requires a specific expression of the connection form $\alpha$ in the neighborhood of the point where the plug is inserted. In this section we present the state of art on the existence of periodic orbits and prove that every manifold of dimension $2n+1>3$ admits a Beltrami field (not volume-preserving) without periodic orbits.

\subsection{The Weinstein conjecture} The Weinstein conjecture states that any Reeb vector field in a closed contact manifold has at least one periodic orbit. The conjecture is known to be true in dimension three \cite{T}, as well as for overtwisted contact structures in any dimension \cite{AH}. Concerning stable Eulerisable structures, it is known to be true in dimension three (where they coincide with stable Hamiltonian structures) in the following form.
\begin{theorem*} [\cite{HT}]
Let $M$ be a closed oriented three manifold with a stable Hamiltonian structure. If $M$ is not a $T^2$-bundle over $S^1$, then its Reeb field has a closed orbit.
\end{theorem*}

A counterexample in the $T^2$-bundle over $S^1$ is provided by taking the mapping torus of an irrational rotation in $T^2$. This counter example generalizes to any dimension for stable Eulerisable structures. Even if we defined stable Eulerisable structures in odd dimensions, since it is the natural set for the study of Beltrami fields, the definition makes sense also in even dimensions.

\begin{Claim}
Let $N$ be a closed manifold of dimensions $n\geq 2$ such that $\chi(N)=0$. Then there is a $N$-bundle over $S^1$ endowed with a stable Eulerisable structure such that its Reeb field has no periodic orbits.
\end{Claim}

\begin{proof}
Following \cite{Pl} and \cite{Wa}, any manifold such that $\chi(N)=0$ admits a volume-preserving diffeomorphism $\varphi:N \rightarrow N$ without periodic points. Consider the suspension of this diffeomorphism, i.e. the manifold $M= N\times I / \sim$ where we identified $(p,0)$ with $(\varphi(p),1)$. If we denote $t$ a coordinate in $I$, it induces a coordinate $\theta$ in $M$. The vector field $X=\pp{}{\theta}$ has no periodic orbits, and preserves a volume form $\mu$ since $\varphi$ was volume-preserving. It is the Reeb field of the stable Eulerisable structure $(d\theta, \iota_X\mu)$.
\end{proof}

The fact that geodesible fields do not admit plugs, as well as the Weinstein conjecture for stable Hamiltonian structures, motivates the idea that some version of the Weinstein conjecture could be true for stable Eulerisable structures in high dimensions.

In the non volume-preserving case, it was proved in \cite{R2} the following positive result, with the assumption that both the vector field and the metric making its orbits geodesics are real analytic.

\begin{theorem*}
Let $M$ be a closed oriented $3$-manifold which is not a torus bundle over the circle. Then any real analytic geodesible (or equivalently Beltrami) field has a periodic orbit.
\end{theorem*}

The smooth case is still open. We will prove in the next subsection that, in the high dimensional setting, there always exist Beltrami fields without periodic orbits.

\subsection{Aperiodic Beltrami fields using round Morse functions}

In \cite{As}, Asimov introduced round handle decompositions and proved that every manifold of dimension at least $4$ satisfying $\chi(M)=0$ admits such a decomposition. This concept was later related to the existence of round Morse functions, introduced by Thurston \cite{Th}.

\begin{defi}
A \textbf{round Morse function} on a smooth manifold $M$ is a function $f:M\rightarrow \R$ such that:
\begin{itemize}
\item the set of critical points of $f$ is a disjoint union of circles,
\item the corank of $f$ in a critical point is $1$.
\end{itemize}

\end{defi}

The existence of such a function is equivalent to the existence of a round handle decomposition, a fact that was rigorously proved by Miyoshi in \cite{Mi}. Miyoshi obtained a round Morse lemma, where one can have standard Morse coordinates or twisted Morse coordinates. However, it is always possible to find a round Morse function without twisted critical circles, and hence we only state the untwisted case.

\begin{lemma}[Untwisted Round Morse Lemma]
Let $f:M^{n+1} \rightarrow \R$ be a round Morse function without twisted singular circles. Then there exist global coordinates $(\theta,x_1,...,x_{n})$ in a neighborhood $U= S^1 \times D^n$ near any critical circle $C$ such that
$$f=-x_1^2-...-x_r^2+x_{r+1}^2+...+x_n^2,$$
where $r$ is the index of the critical circle.
\end{lemma}

The well known relation between round Morse functions and Morse-Smale flows provides a starting point to construct aperiodic Beltrami flows.

\begin{proof}[Proof of Theorem \ref{aperio}]

 Taking into the account previous discussions, we only need to construct a vector field $X$ satisfying the following three properties:
\begin{enumerate}
\item There is a one form $\alpha$ such that $\alpha(X)>0$ and $\iota_X(d\alpha)^n=0$,
\item $X$ has a finite number of periodic orbits,
\item for every periodic orbit $\gamma$, there is a point $p\in \gamma$  and a coordinate $z$ in a neighborhood $U$ of $p$ such that $X|_U=\pp{}{z}$ and $\alpha|_U=dz$.
\end{enumerate}
If we achieve this, inserting a plug in each neighborhood of the point $p$ of each periodic orbit yields a vector field without periodic orbits and a one form $\beta$ such that $\beta(X)>0$ and $\iota_X(d\beta)^n=0$ and hence by Lemma \ref{beltchar} the vector field $X$ is a Beltrami field.

\paragraph{\textbf{Construction of a vector field satisfying (1)-(3).} \newline} 

Take a round Morse function $f$ without twisted components and a metric which is "nice" in the neighborhood of the finite amount of critical circles: i.e. it looks like the standard metric in $S^1 \times D^{2n}$ with the round Morse Lemma coordinates. Then the gradient defined by $f$ is the vector field satisfying $g(X,\cdot)=df$, and is a vector field without periodic orbits but with the set of fixed points being the critical circles of $f$. We will do a modification of this vector field around the critical circles to obtain a vector field with a finite amount of periodic orbits and construct a one form satisfying $(1)$ and $(3)$. Consider one of the critical circles $\gamma$ of $f$, let us assume that it is a maximum, since everything works analogously on each critical circle.

\textbf{Step 1: around the orbit.}

In the neighborhood $U=S^1 \times D^{2n}$ with coordinates $(\theta,x_1,...,x_{2n})$, we assume that the metric is the standard $g=d\theta^2+ \sum_{i=1}^{2n} dx_i^2$. Hence the gradient of $f$ has the following expression.
$$\operatorname{grad}(f)=\sum_{i=1}^{2n} x_i \pp{}{x_i},$$
where $f=\sum_{i=1}^{2n} x_i^2$. Take the function $\rho= \sum_{i=1}^{2n} x_i^2$ (independently of the index of the critical circle of $f$), and $\varphi(\rho)$ is a bump function which is constantly equal to $1$ around $\rho=0$ and $0$ around $\rho=1$. We can now modify the gradient of $f$, taking instead 
$$ X= \varphi(\rho) \pp{}{\theta} + \operatorname{grad}(f), $$
which has a single periodic orbit at $S^1 \times \{0 \}$.
Construct the one form 
$$ \alpha_\gamma= \varphi(\rho) d\theta + df, $$
which satisfies $\alpha_\gamma(X)>0$, $\alpha_\gamma|_{\partial U}=df$. Computing its exterior derivative we have
$$ d\alpha_\gamma= \varphi'(\rho) d\rho \wedge d\theta, $$
which satisfies $(d\alpha_\gamma)^2=0$. In a small neighborhood $V$ of the orbit where $\varphi(\rho) \equiv 1$ we have $\alpha_\gamma= d\theta + df$. The form $\alpha_\gamma$ extends outside of $U$ as $df$. Denote $\alpha$ the one-form which is $df$ outside the neighborhoods of the critical circles and the constructed $\alpha_\gamma$ on them. Doing this at every orbit, we obtain a vector field $X$ with a finite number of periodic orbits and a one form $\alpha$ satisfying $\alpha(X)>0$ and $(d\alpha)^2=0$. Only condition (3) is left to check.  Figure \ref{circle} depicts schematically the modification for a critical circle with arbitrary index.
\begin{center}
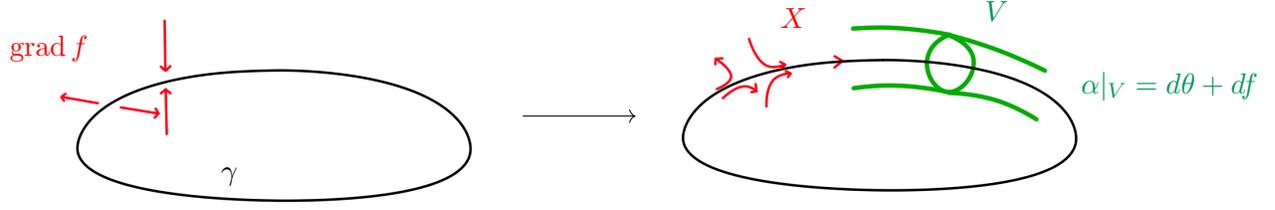
\begin{figure}[!ht]
	\begin{tikzpicture}[scale=1]
     \node[anchor=south west,inner sep=0] at (0,0) {\includegraphics[scale=0.22]{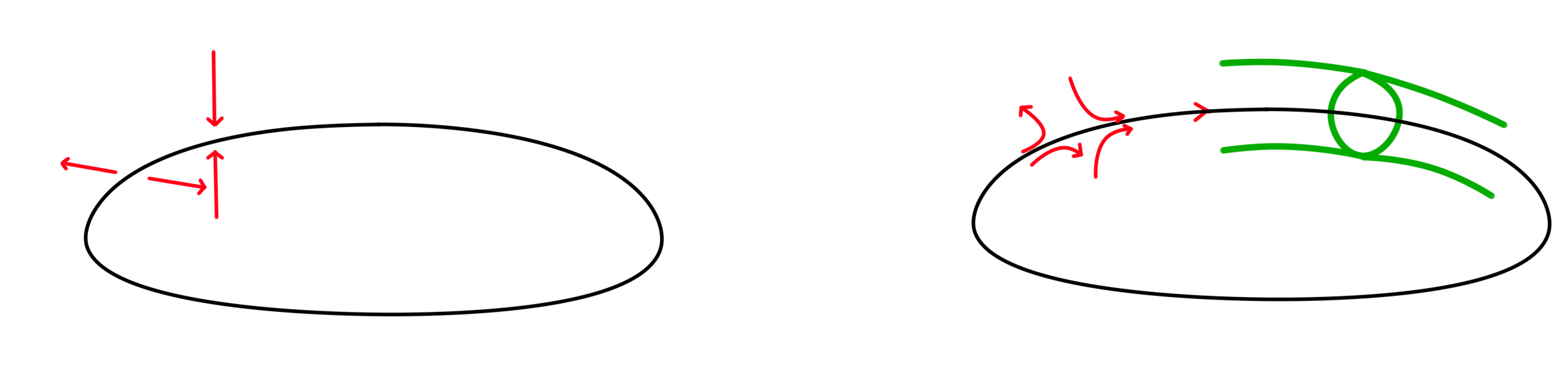}};
\node[red] at (0.4,2.5,0) {$\operatorname{grad} f$};
\node[red] at (10.3,2.9,0) {$X$};
\node[ForestGreen] at (13,3,0) {$V$};
\node[ForestGreen] at (15.3,2,0) {$\alpha|_V=d\theta + df$};

\node at (2.8,0.8,0) {$\gamma$};

\draw [->] (6.7,1.6,0)--(8.2,1.6,0);

   \end{tikzpicture}
\caption{Modification around circle}
\label{circle}
\end{figure}
\end{center}

\textbf{Step 2: around a point.}

As mentioned in the previous step, in the neighborhood $V$ of the orbit we can now assume $X=\pp{}{\theta} + \operatorname{grad}f$ and $\alpha= d\theta + df$. Around a point $p$ in the orbit $S^1\times D^{2n}$, the $S^1$-coordinate $\theta$ defines a function $z$. Hence there are coordinates $(z,x_1,...,x_{2n})$ on a small neighborhood $U$ such that $X|_U= \pp{}{z} + \sum_{i=1}^{2n} x_i \pp{}{x_i}$ and $\alpha|_U= dz + df$. By taking the neighborhood small enough, we can assume that there is a function $h$ such that $X=\pp{}{h}$ by the flow box theorem. Denote $r=z^2 + x_1^2 + ...+ x_{2n}^2$. Take $\varphi(r)$ a bump function which is $1$ around $0$ and vanishes around $1$. Construct the one form
$$  \beta =  \varphi(r) dh + (1-\varphi(r))[dz + df]. $$
Contracting it with $X$ we have that
\begin{align*}
 \beta (X)&= \varphi(r) dh(X) + (1-\varphi(r))[dz+df](X) \\
&= \varphi(r) dh(\pp{}{h}) + (1-\varphi(r)) dz(\pp{}{z} + \operatorname{grad}f) + (1-\varphi(r)) df(\pp{}{z} + \operatorname{grad}f)  \\
&= \varphi(r) + (1-\varphi(r))[1+ df(\operatorname{grad} f)].
\end{align*} 
Since $df(\operatorname{grad} f)$ is positive except at $r=0$, we deduce that $ \beta(X)>0$. Furthermore, we have $d \beta= \varphi' dr \wedge dh - \varphi' dr \wedge dz - \varphi' dr \wedge df$ which implies $(d \beta)^2=0$ and coincides with $dz+df$ on $\{r=1\}$. Denote again $\alpha$ the form $\beta$ extended as $\alpha$ outside the neighborhood $U$.
Hence in a neighborhood of $r=0$ where $\varphi(r)\equiv 1$ we have that $X|_U = \pp{}{h}$ and $\beta|_U=dh$. The orbit through $r=0$ is the isolated periodic orbit.  Figure \ref{point} depicts schematically the modification for a critical circle with arbitrary index.
\begin{center}
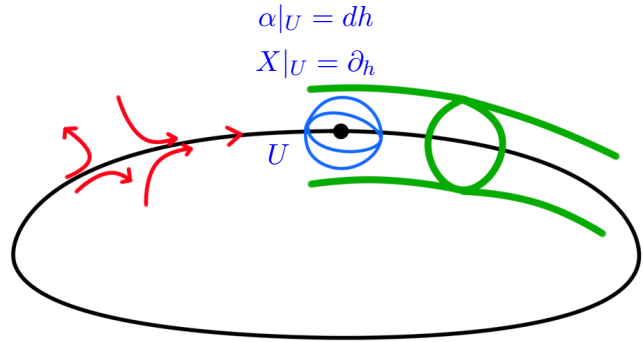
\begin{figure}[!ht]
	\begin{tikzpicture}[scale=1]
     \node[anchor=south west,inner sep=0] at (0,0) {\includegraphics[scale=0.35]{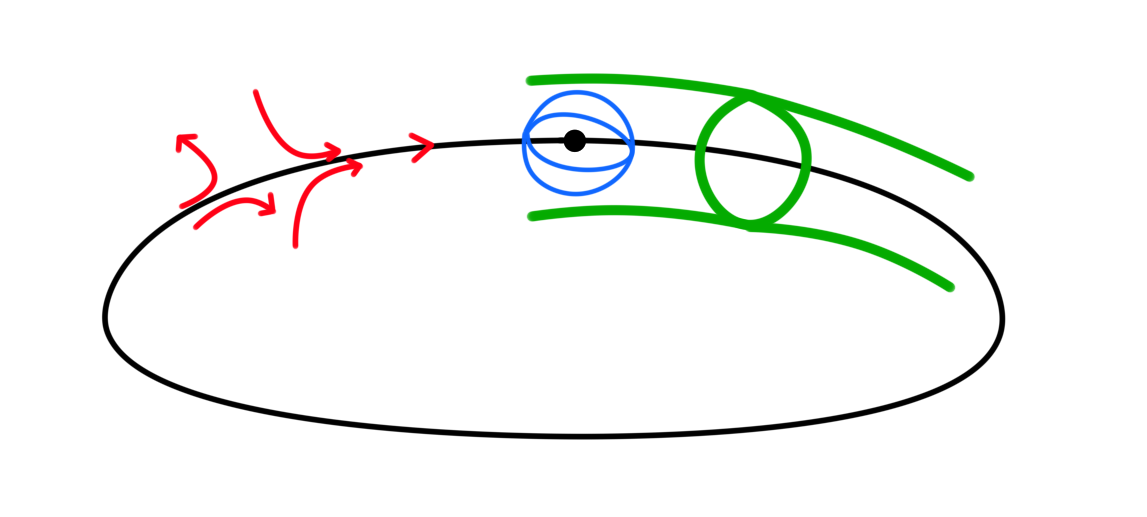}};
     \node[blue] at (5,5,0) {$\alpha|_U=dh$};
     \node[blue] at (5,4.4,0) {$X|_U=\partial_h$};
     \node[blue] at (4.5,3.2,0) {$U$};

   \end{tikzpicture}
\caption{Modification around point}
\label{point}
\end{figure}
\end{center}
In particular, condition (3) is satisfied for the closed orbit $\gamma$.

Doing this at every critical circle, we prove that the pair $(X,\alpha)$ satisfies the conditions (1)-(3), which proves the theorem.
\end{proof}

\begin{Remark}
Note that the constructed aperiodic Beltrami fields are furthermore irrotational. Since $(d\beta)^n\equiv 0$, their curl vanishes everywhere.
\end{Remark}

\end{document}